\documentclass[11pt]{amsart}

\usepackage{amsmath,amssymb,amsthm}
\usepackage{graphicx}
\usepackage{cite}
\usepackage[usenames, dvipsnames]{color}
\usepackage{hyperref}

\normalsize

\newtheorem*{maincor-repeat}{Corollary \ref{c:st-1}}

\def\dfrac#1#2{\lower0.15ex\hbox{\large$\frac{#1}{#2}$}}

\allowdisplaybreaks

\def\eps{\varepsilon}

\def\Z{\mathbb{Z}}
\def\R{\mathbb{R}}
\def\C{\mathbb{C}}

\def\F{\mathbb{F}}

\def\E{\mathbb{E}}

\DeclareMathOperator\spa{span}
\DeclareMathOperator\rank{rank}

\DeclareMathOperator\Sym{Sym}
\DeclareMathOperator\im{Im}
\DeclareMathOperator\id{id}
\newcommand{\vect}[1]{\boldsymbol{#1}}

\DeclareMathOperator\Af{A4}
\DeclareMathOperator\Ai{A5}

\newtheorem{firstthm}{Proposition}[section]
\newtheorem{thm}[firstthm]{Theorem}
\newtheorem{prop}[firstthm]{Proposition}

\newtheorem{lemma}[firstthm]{Lemma}
\newtheorem{cor}[firstthm]{Corollary}

\newtheorem{conj}[firstthm]{Conjecture}

\newtheorem{ques}[firstthm]{Question}

\theoremstyle{definition}
\newtheorem{example}[firstthm]{Example}
\newtheorem{remark}[firstthm]{Remark} 
\newtheorem{obs}[firstthm]{Observation} 
\newtheorem{defn}[firstthm]{Definition}

\title{Local aspects of the Sidorenko property for linear equations}

\author{Daniel Altman}

\begin{document}

\maketitle

\begin{abstract}
A system of linear equations in $\F_p^n$ is \textit{Sidorenko} if any subset of $\F_p^n$ contains at least as many solutions to the system as a random set of the same density, asymptotically as $n\to \infty$. A system of linear equations is \textit{common} if any 2-colouring of $\F_p^n$ yields at least as many monochromatic solutions to the system of equations as a random 2-colouring, asymptotically as $n\to \infty$. Both classification problems remain wide open despite recent attention. 

We show that a certain generic family of systems of two linear equations is not Sidorenko. In fact, we show that systems in this family are not locally Sidorenko, and that systems in this family which do not contain additive tuples are not weakly locally Sidorenko. This endeavour answers a conjecture and question of Kam\v cev--Liebenau--Morrison. Insofar as methods, we observe that the true complexity of a linear system is not maintained under Fourier inversion; our main novelty is the use of higher-order methods in the frequency space of systems which have complexity one. We also give a shorter proof of the recent result of Kam\v cev--Liebenau--Morrison and independently Versteegen that any linear system containing a four term arithmetic progression is uncommon.
\end{abstract}

\section{Introduction}\label{s:intro}

\subsection{Context}
In graph theory, Sidorenko's conjecture \cite{S93} predicts that if $H$ is a bipartite graph, then among all graphs $G$ with $n$ vertices and fixed average degree, the number of copies of $H$ in $G$ is asymptotically (as $|V(G)|\to \infty$) minimised when $G$ is random. We will call graphs $H$ which satisfy this property \textit{Sidorenko}. Sidorenko's conjecture has been resolved for certain classes of bipartite graph (for example trees, cycles, complete bipartite graphs \cite{S91}, see also \cite{CFS10, H10, LS11, S14, KLL16, CL17, CKLL18, CL21}) though the conjecture in general remains open, and is indeed of great interest. There is a related colouring property which is also of interest: for which graphs $H$ is it true that, among all two-colourings of $K_n$, the number of monochromatic copies of $H$ is minimised by a random two-colouring (asymptotically as $n\to \infty$)?  Such graphs $H$ are called \textit{common}. It is not difficult to see that if a graph satisfies the Sidorenko property, then it is common. Goodman showed that $K_3$ is common \cite{G59}. Erd\H os \cite{E62} subsequently conjectured that $K_4$ is common and Burr--Rosta \cite{BR80} that every graph is common. Sidorenko \cite{Sid89} showed that a triangle with a pendant edge is not common and Thomason \cite{T89} showed that in fact $K_4$ is not common, together disproving both conjectures. In fact, any graph which contains $K_4$ is uncommon \cite{JST96}. It is not difficult to see that if $H$ is \textit{not} bipartite then it is not Sidorenko, and so we have a conjectural classification of all graphs. Even a reasonable conjectured classification of common graphs seems out of reach at the moment.

There has been recent interest in obtaining analogous conjectures and results in arithmetic contexts, initiated by Saad and Wolf \cite{SW17}. Which systems of linear forms $\Psi$ have the property that, among all subsets $S$ of $\F_p^n$ of fixed size, the number of solutions to $\Psi$ in $S$ is minimised when $S$ is chosen randomly? Such systems are called Sidorenko. Those systems with the property that random two-colourings minimise the number of monochromatic solutions among all two-colourings are called common. Again, Sidorenko implies common. In the arithmetic setting, neither property has a conjectured classification. 

One situation in which a classification is known is when the system of linear forms has codimension one (i.e. its image is described by a single equation). Here, an equation is Sidorenko if \cite{SW17} and only if \cite{FPZ21} its coefficients may be partitioned into pairs which sum to zero, and is common and not Sidorenko if and only if it has an odd number of variables. For higher codimension systems, one makes the trivial observation that the union of two Sidorenko systems on disjoint variable sets is Sidorenko. Kam\v cev, Liebenau and Morrison give a nontrivial method by which higher codimension Sidorenko systems can be built from Sidorenko equations; see \cite[Theorem 1.3]{KLMSid}. There are also a number of results in the negative direction \cite{KLMCom, KLMSid, V4ap, SW17}. Finally, we note that Kr\'al'--Lamaison--Pach have reported on a partial classification of common systems comprising two equations in $\F_2^n$ \cite{KLP22} and have since obtained a near-complete classification of these systems, which is due to appear in forthcoming work (personal communications).

Lov\'asz  \cite{L10} showed that the graph-theoretic Sidorenko conjecture is locally true: for all bipartite graphs $H$, there is no suitably small perturbation of the random graph $G$ which increases the number of copies of $H$ in $G$. He subsequently classified locally Sidorenko graphs: a graph $H$ is locally Sidorenko if and only if it is a forest or has even girth \cite[Theorem 16.26]{L12} (see also \cite{FW17} to this end). The study of local notions of commonness has also attracted recent interest: \cite{L12, CHL22, HKKV22}. One of the purposes of this article is to introduce these local notions to the arithmetic setting.   Finally, to further motivate the study of these local properties, we note that \cite{Alt22} uses a local Sidorenko property of a particular linear system as a key ingredient in its answer to a question of Alon.

\subsection{Results}
After Fox--Pham--Zhao \cite{FPZ21} completed the classification of codimension one systems, the next natural goal ought perhaps to be a classification of codimension two systems. Some results on codimension two systems have been obtained in \cite[Section 4]{KLMSid} and \cite{KLP22}. 

It is trivially the case that codimension one systems are Sidorenko/common if and only if they are locally Sidorenko/common; local notions are only interesting when the codimension is at least two. We study two different notions of locality. Loosely speaking, the weaker notion (`weakly locally Sidorenko/common') allows the radius of perturbation to depend on the function by which we perturb, and the stronger notion (`locally Sidorenko/common') asks that the radius is uniform in the choice of function. 

One of our main theorems shows that a certain generic family of systems of codimension 2 is not Sidorenko, and (nearly) classifies this family by whether they are weakly locally Sidorenko or locally Sidorenko. Let $\Psi=(\psi_1,\ldots,\psi_t)$ be a system of linear forms, each mapping $\F_p^D$ to $\F_p$, whose image is determined by a $2 \times t$ system of linear equations with coefficients 
 \[M:=\begin{pmatrix}
  b_{11} & b_{12} & \cdots & b_{1t} \\ 
  b_{21} & b_{22} & \cdots & b_{2t}
\end{pmatrix}.\]
We will call $\Psi$ \textit{linearly generic} if every $2 \times 2$ minor of $M$ has nonzero determinant. An \textit{additive $k$-tuple} is a tuple $(x_1, x_2, \ldots, x_k)$ satisfying the equation $x_1-x_2+x_3-x_4 + \cdots + (-1)^{k+1}x_k=0$. We direct the reader to Section 2 and in particular Definition \ref{d:locally-sid} and Definition \ref{d:weakly-sid} for formal definitions of the weak local Sidorenko property and the local Sidorenko property. We note that Sidorenko implies locally Sidorenko, which in turn implies weakly locally Sidorenko.

\begin{thm}\label{t:classification}
Let $\Psi$ be a rank two system of linear equations in $t$ variables which is linearly generic. Then $\Psi$ is not locally Sidorenko (and so not Sidorenko). Furthermore, if $\Psi$ is weakly locally Sidorenko then $t$ is odd and if furthermore $p$ is sufficiently large in terms of $t$ then $\Psi$ contains an additive $(t-1)$-tuple. 
\end{thm}

We have not classified systems with $t$ odd and which contain an additive $(t-1)$-tuple by whether they are weakly locally Sidorenko, but remark that there does exist such a system which is weakly  locally Sidorenko (see Example \ref{ex:weak-not-local}).

As a consequence of Theorem \ref{t:classification}, we answer the following conjecture and question of Kam\v cev--Liebenau--Morrison. In the following conjecture (respectively, question), the condition that $s(\Psi)=4$ (resp. $s(\Psi)=t-1$)  asks that the projection of the image of $\Psi$ onto any 3 (resp. $t-2$) coordinates contains a coordinate hyperplane; for the meantime it can be thought of as a technical condition which implies that $\Psi$ is linearly generic.

\begin{conj}[{\cite[Conjecture 5.1]{KLMSid}}]\label{c:KLM5sid}
Let $\Psi$ be a system of $5$ linear forms whose image has codimension 2. If $s(\Psi) = 4$ then $\Psi$ is not Sidorenko. 
\end{conj}

\begin{ques}[{\cite[Question 5.3]{KLMSid}}]\label{q:KLMkSid}
Let $t\geq 7$ be odd. Does there exists a system of $t$ linear forms $\Psi$ whose image has codimension 2 and with $s(\Psi)=t-1$ which is Sidorenko? 
\end{ques}

We have the following corollary of Theorem \ref{t:classification}.

\begin{cor}\label{c:st-1}
Conjecture \ref{c:KLM5sid} is true and the answer to Question \ref{q:KLMkSid} is no. 
\end{cor}

Define the \textit{complexity} of $\Psi=(\psi_1,\ldots,\psi_t)$, where each $\psi_i$ is a linear map from $\F_p^D$ to $\F_p$, to be the smallest positive integer $s$ such that the functions $(\psi_i^s)_{i=1}^t$ are linearly independent over $\F_p$. Loosely speaking, if the complexity of a system is 1, then its analysis should be controlled by Fourier analysis, and one ought not require the higher-order theory. We direct the reader to \cite{GW10}, \cite{GW11}, \cite{GW11GAFA} for relevant definitions, discussions and results to this end. One may easily compute that for $t\geq 5$, systems whose images have codimension two generically have complexity one.  It is not difficult, using higher-order constructions, to contrive complexity $\geq 2$ systems which are weakly locally Sidorenko but not Sidorenko. The situation is less clear for complexity one systems. We believe therefore that the following corollary of Theorem \ref{t:classification} is also worth noting. 

\begin{cor}\label{cor:weak-not-local}
Among complexity one systems, the weak local Sidorenko property is not equivalent to the local Sidorenko property.
\end{cor}

The explicit example demonstrating Corollary \ref{cor:weak-not-local} is Example \ref{ex:weak-not-local} in Section \ref{s:examples}.

Perhaps of independent interest is the observation and method underlying Theorem \ref{t:classification} and its corollaries. We observe that although these systems have complexity one, their `dual' system (i.e. the linear system obtained in frequency space after Fourier inversion) need not have complexity one. The most involved part of Theorem \ref{t:classification} utilises higher-order constructions in the frequency space.

We also provide the first example of a system which is locally Sidorenko but not Sidorenko; see Example \ref{ex:sid-not-local}. (This is the same example as in \cite{Alt22}, though there we prove a local Sidorenko property with respect to the $\ell^\infty$ norm on the Fourier side.)

\begin{prop}\label{p:local-not-sid}
In the arithmetic setting, the Sidorenko property is not equivalent to the local Sidorenko property.
\end{prop}
This result is perhaps slightly less trivial than it sounds; for example in the graph-theoretic setting, it is not known whether local commonness is equivalent to commonness (see \cite{CHL22}). One of the reasons we have chosen to study these local properties is that all existing proofs in the literature that systems are not Sidorenko \cite{KLMSid, V4ap, SW17, FPZ21} prove the stronger statement that in fact these systems are not locally Sidorenko. Finally, in the following proposition, we also note that the linear genericity condition in Theorem \ref{t:classification} is necessary. 

\begin{prop}\label{p:not-generalise}
There exists a codimension two system which is locally Sidorenko, not linearly generic and which does not contain an additive tuple.
\end{prop}

The proof of Proposition \ref{p:not-generalise} is conducted as Example \ref{ex:no-add-tup}. 

In another direction, it has been known for some time \cite[Theorem 12]{JST96} that any graph containing $K_4$ is uncommon. Saad and Wolf \cite{SW17} asked whether the same is true for four-term arithmetic progressions in the arithmetic context. This question was answered in the affirmative recently by Kam\v cev--Liebenau--Morrison \cite{KLMCom} and independently by Versteegen \cite{V4ap}. Both proofs build on a `quadratic' construction of Gowers \cite{G20}. In Section \ref{s:local-4ap}, we give (together with a lemma from \cite{KLMCom}) a short proof of Saad and Wolf's conjecture which does not rely on higher-order constructions. Our proof is also more amenable to generalisation to rank two systems in $2k$ variables, which would be a significant step in the classification of weakly locally common systems\footnote{We do not consider weak local commonness in this document and so we do not give this concept a formal definition here. However, we hope that the enthusiastic reader may formulate a definition from that of the weak local Sidorenko property Definition \ref{d:weakly-sid} and the common property Definition \ref{d:common}.}. 

We set up some notational and technical preliminaries in Section \ref{s:prelim}. We prove Theorem \ref{t:classification} in Section \ref{s:generic-rank2}. In Section \ref{s:examples} we give examples to prove Corollary \ref{cor:weak-not-local}, Proposition \ref{p:local-not-sid} and Proposition \ref{p:not-generalise}. Finally in Section \ref{s:local-4ap} we give a new proof of the answer to Saad and Wolf's question \cite[Question 3.6]{SW17}.
\newline\newline
\textbf{Acknowledgments.} The author thanks Ben Green for drawing his attention to the study of the local Sidorenko property in the graph-theoretic setting (in particular \cite{L10}), and for valuable feedback on an earlier version of this document.

\section{Preliminaries}\label{s:prelim}

\subsection{Definitions, basic properties}\label{ss:defn}
Throughout we let $\Psi = (\psi_1,\ldots, \psi_t)$ be a system of $\F_p$-linear forms, each mapping $\F_p^D$ to $\F_p$; we often abuse notation and let them induce linear maps from $(\F_p^n)^D$ to $\F_p^n$ in the natural way.  Define the operator $T_\Psi$  acting on functions  $f: \F_p^n \to \C$ by the formula
\[T_\Psi(f) := \E_{x_1,\ldots,x_D\in \F_p^n}f(\psi_1(x_1,\ldots,x_D))\cdots f(\psi_t(x_1,\ldots,x_D)),\] 
where throughout $\E$ is a normalised count, so here is shorthand for $\frac{1}{p^{nD}}\sum$. We will always assume that $D \leq t$ and that $\Psi^{-1}(0)=0$ (it is clear that one may reduce to this case).  At times it will be more convenience to describe a system of linear forms by the equations which define its image. Suppose that $M_\Psi$ is an $\F_p$-valued matrix such that $\im \Psi = \ker M_\Psi$. Then we have 
\[ T_\Psi(f) = \E_{\vect x \in \ker M_\Psi}f(x_1)\cdots f(x_t),\]
where $\vect x = (x_1,\ldots,x_t)\in (\F_p^n)^t$.

Throughout this document we will deal with functional versions of the Sidorenko and common properties in which we ask the following of all functions $f:\F_p^n \to [0,1]$, rather than just restricting to characteristic functions of sets:

\begin{defn}\label{d:sid}
A system of $t$ linear forms $\Psi$ is \textit{Sidorenko} if, for all $n\geq 1$ and all $f:\F_p^n \to [0,1]$, we have
\[ T_\Psi(f)  \geq (\E_{x \in \F_p^n}f(x))^t.\]
\end{defn}

Note that any function can be written as the sum of a constant function and a function with zero average, whence the inequality in Definition \ref{d:sid} can be written as $T_\Psi(\alpha + f)  \geq   \alpha^t$,
where the condition must hold for all $\alpha \in [0,1]$ and all $f: \F_p^n \to [-\alpha, 1-\alpha]$ with $\E_x f(x) = 0$. A system is locally Sidorenko if the above inequality is satisfied for all sufficiently small perturbations of the constant function. 

\begin{defn}\label{d:locally-sid}
A system of $t$ linear forms $\Psi$ is \textit{locally Sidorenko} if for all $\alpha \in [0,1]$ there exists $\eps>0$ such that for all $n \geq 1$ and all $f:\F_p^n\to [-\alpha, 1-\alpha]$ with $\E_xf(x) = 0$, we have 
\[ T_\Psi(\alpha + \eps f) \geq \alpha^t.\]
\end{defn}

We remark that that this notion of locality considers small perturbations of the constant function with respect to the $\ell^\infty$ norm on the space of functions. One may also consider analogous notions of locality for different norms. We do not do so in this document, but note that for example the proof of the main result in \cite{Alt22} proves and uses the fact that a certain system of forms is locally Sidorenko with respect to the $\ell^\infty$ norm on the Fourier side. 

Above, the size of the neighbourhood of the constant function is fixed independently of $n,f$. If the size of the neighbourhood is allowed to vary with these parameters, then we will say that the linear system is weakly locally Sidorenko.

\begin{defn}\label{d:weakly-sid}
A system of $t$ linear forms $\Psi$ is \textit{weakly locally Sidorenko} if for all $n\geq 1$, for all $\alpha \in [0,1]$ and all $f:\F_p^n\to [-\alpha, 1-\alpha]$ with $\E_xf(x) = 0$, there exists $\eps_f >0$ such that for all $\eps < \eps_f$ we have
\[ T_\Psi(\alpha + \eps f) \geq \alpha^t.\]
\end{defn}

\begin{defn}\label{d:common}
A system of $t$ linear forms $\Psi$ is \textit{common} if, for all $n\geq 1$ and all $f:\F_p^n \to [0,1]$, we have
\[ T_\Psi(f) + T_\Psi(1-f)  \geq 2^{1-t}.\]
\end{defn}

Local commonness and weak local commonness may be defined analogously.

Let $\Psi=(\psi_1,\ldots, \psi_t)$ be a system of $t$ linear forms. Let $\alpha$ be a constant and $f$ a function on $\F_p^n$ with $\E_x f(x) = 0$. For a subset $S\subset [t]$, we denote the corresponding subsystem of $\Psi$ by $\Psi(S)$. Then, by the multilinearity of the $T$ operator we have
\begin{equation}\label{e:expansion}
T_\Psi(\alpha + f) = \sum_{S \subset [t]} \alpha^{t-|S|}T_{\Psi(S)}(f).
\end{equation}
Note that if $f$ has $\E_xf(x)=0$ and if the image of $\Psi(S)$ contains a coordinate hyperplane then $T_{\Psi(S)}(f) =0$. With some linear algebra, one sees that if a codimension two system is linearly generic (recall: every $2\times 2$ minor of the matrix defining its equations has nonzero determinant), then for all nonempty $S \subset [t]$ with $|S| \leq t-2$, the subsystem $\Psi(S)$ contains a coordinate hyperplane and so yields $T_{\Psi(S)}(f)=0$. Thus for linearly generic systems we have 
\begin{equation}\label{e:lingen-expansion}
T_\Psi(\alpha + f) = \alpha^t + \alpha \sum_{|S|=t-1} T_{\Psi(S)}(f) + T_{\Psi}(f).
\end{equation}
Furthermore, for such systems, if $|S|=t-1$ then the image of $\Psi(S)$ has codimension one (i.e. its image is described by a single equation).

\subsection{Fourier inversion}\label{ss:fourier}
Recall the Fourier transform defined by \[\hat f(h) = \E_{x \in \F_p^n}f(x)e_p(-h \cdot x),\] where $e_p(\cdot) := e^{\frac{2\pi i \cdot}{p}}$. Recall also the Fourier inversion formula: \[f(x) = \sum_{h\in \F_p^n}\hat f(h)e_p(h\cdot x),\]
and Parseval's identity
\[\E_x |f(x)|^2 = \sum_h |\hat f(h)|^2.\]
For a linear system $\Psi$, let $M=M_\Psi$ have $\im \Psi = \ker M$. We will make use of the following consequence of the Fourier inversion formula:
\begin{equation}\label{e:T-inversion}
T_\Psi(f) = \sum_{\vect h \in \im M^T} \hat f(h_1)\cdots \hat f(h_t),
\end{equation}
where $\vect h = (h_1,\ldots, h_t)\in (\F_p^n)^t$.

\subsection{Tensoring}\label{ss:tensor}

\begin{defn}
Let $f_1:\F_p^{n_1} \to \C$ and $f_2: \F_p^{n_2} \to \C$. Then the tensor product $f:= f_1 \otimes f_2 : \F_p^{n_1+n_2} \to \C$ is the function defined by $f(x) = f_1(\pi_1(x))f_2(\pi_2(x))$, where $\pi_1$ is the projection onto the first $n_1$ coordinates of $\F_p^{n_1+n_2}$ and $\pi_2$ is the projection onto the last $n_2$ coordinates. 
\end{defn}

It is clear that $\otimes$ is an associative operation and so one may naturally define the tensor product of a finite collection of functions.

\begin{obs}\label{o:tensor-mult}
For a system of linear forms $\Psi$, we have
\begin{equation}\label{e:tensor}
T_\Psi(f_1 \otimes f_2) = T_\Psi(f_1)T_\Psi(f_2).
\end{equation}
\end{obs}

Note that there is a slight abuse of notation here: we interpret $T_\Psi$ as an operator on $\F_p^{n_1+n_2}$, $\F_p^{n_1}$ and $\F_p^{n_2}$ as it appears respectively. The equation above is easily verified by direct computation.

\section{Proof of Theorem \ref{t:classification}}\label{s:generic-rank2}

Recall that we call a system of two equations linearly generic if 
the $2 \times t$ matrix \[M:=\begin{pmatrix}
  b_{11} & b_{12} & \cdots & b_{1t} \\ 
  b_{21} & b_{22} & \cdots & b_{2t}
\end{pmatrix}\]
which defines the equations has the property that  all of its $2 \times 2$ minors have nonzero determinant.
 
\subsection{Systems without additive tuples are not weakly locally Sidorenko}
We firstly work towards proving the following theorem, which is one of the statements in Theorem \ref{t:classification}.
\begin{thm}\label{t:weakSid}
Let $\Psi$ be a system of two linear equations in $t$ variables which is linearly generic. If $t$ is even or if $\Psi$ does not contain an additive tuple and $p$ is sufficiently large in terms of $t$, then $\Psi$ is not weakly locally Sidorenko.
\end{thm}

We now prove the theorem, modulo the proof of two propositions which will appear subsequently and occupy the remainder of the subsection. In both propositions, we prove something more general than what is needed for the application in this document. We do so to illustrate the scope of the constructions and because it is clear that these more general statements can be used to obtain more general results. See also Remark \ref{r:relax}.  

\begin{proof}[Proof of Theorem \ref{t:weakSid}]
Recall from Equation (\ref{e:lingen-expansion}) that we have 
\[T_\Psi(\alpha + \eps f) = \alpha^t + \alpha \eps^{t-1}\sum_{|S|=t-1} T_{\Psi(S)}(f) + \eps^t T_\Psi(f).\]
It suffices to show that there exists $f$ with $\E f = 0$ such that $\sum_{|S|=t-1} T_{\Psi(S)}(f)<0$; then one makes $\eps$ appropriately small to conclude. If $t-1$ is odd, then this is trivial: take $f$ for which $\sum_{|S|=t-1} T_{\Psi(S)}(f) \ne 0$ (it is not difficult to see that such an $f$ must exist), and if $\sum_{|S|=t-1} T_{\Psi(S)}(f)  > 0$ then replace $f$ with $-f$. Henceforth we assume that $t$ is odd.

Our strategy will be to find, for each $S \subset [t]$ with $|S|=t-1$, a function $f_S$ with $\E f_S=0$ such that $T_{\Psi(S)}(f_S)=0$ and $T_{\Psi}(f_S) \ne 0$. Then we let $f = \bigotimes_{|S|=t-1} f_S$ and by (\ref{e:tensor}) we see that 
\[T_\Psi(\alpha + \eps f) = \alpha^t +  \eps^t T_\Psi(f),\]
where $T_\Psi(f) = \prod_{|S|=t-1} T_{\Psi}(f_S)\ne 0$. Recall that $t$ is odd so potentially replacing $f$ with $-f$ depending on the sign of $T_\Psi(f)$, we may conclude that for all $\eps$ sufficiently small (indeed, for all $\eps$), we have $T_\Psi(\alpha + \eps f) < \alpha^t$, so $\Psi$ is not weakly locally Sidorenko.

It remains to prove the existence of the functions $f_S$. Recall that each of the systems $\Psi(S)$ has codimension one, so is described by a single equation. This equation cannot be of the form $\sum_{i} (-1)^i x_i =0$ by our assumption that $\Psi$ does not contain an additive tuple. We prove the existence of such an $f_S$ in Proposition \ref{p:generic-linear} and Proposition \ref{p:special}. The first deals with the easier case in which there is a pair of coefficients whose ratio is not equal to $\pm 1$, and the second deals with the remaining cases which are not additive tuples.
\end{proof}

\begin{prop}\label{p:generic-linear}
Let $\Psi$ be a codimension $2$ linear system in $t$ variables which is linearly generic.  Let $\Phi$ be determined by a single equation: $\sum_{i=1}^{l}a_ix_i=0$, where each $a_i$ is nonzero and such that there are $i,j$ such that $a_i \ne \pm a_j$. For all $n\geq 2$, there exists $f: \F_p^n \to [-1,1]$ with $\E_x f (x)=0$, with $T_{\Phi}(f) = 0$ and with $|T_\Psi(f)| \gg_t 1$.
\end{prop}
\begin{proof}
Without loss of generality  $a_1,a_2$ are such that $a_1\ne \pm a_2$. We construct $f$ from its Fourier coefficients. Let $M=M_\Psi=(b_{ij})$ be the $2 \times t$ matrix whose coefficients correspond to the two linear equations of $\Psi$. Using (\ref{e:T-inversion}) we have that 
\begin{equation}\label{e:ts-linear}
T_{\Phi} (f) = \sum_{h \in \F_p^n} \prod_{i=1}^{l} \hat f(a_ih),
\end{equation} and
\begin{equation}\label{e:t-linear}
 T_\Psi(f) = \sum_{h_1,h_2} \prod_{i=1}^{t}\hat f (b_{1i}h_1 + b_{2i}h_2).
 \end{equation}
We claim that if $h_1,h_2$ are linearly independent, then $u:=  (u_1,\cdots, u_t) :=  M^T\binom{h_1}{h_2}$ satisfies $u_i \ne 0$ for all $i$ and $a_1u_i \ne \pm a_2u_j$ for all $i,j$. Indeed, that $u_i \ne 0$ follows from linear independence the fact that $M$ cannot have a zero column, and if $a_1u_i = \pm a_2u_j$, then by the linear independence of $h_1,h_2$ we have $a_1\binom{b_{1i}}{b_{2i}} =\pm a_2\binom {b_{1j}}{b_{2j}}$ which contradicts the fact that all $2\times 2$ minors of $M$ are full rank. Choose any two linearly independent $h_1,h_2$ and define $u_1,\ldots,u_t$ as above.

Now define $\hat f(\pm u_i)=1/(2t)$ for all $i$ and $\hat f(h)=0$ otherwise so that $f$ is clearly real-valued and takes values in $[-1,1]$ by Fourier inversion and the triangle inequality. Since $u_i\ne 0$ for all $i$ we have $\hat f (0) = \E_xf(x) = 0$. Next, since $a_1 \ne \pm a_2$ and since $a_1u_i \ne \pm a_2u_j$ for all $i,j$ we have that at least one of $\hat f (a_1h), \hat f (a_2h)$ is zero for all $h$ and so $T_{\Phi}(f)=0$ from (\ref{e:ts-linear}). Finally, $T_\Psi(f) \geq 2/(2t)^t$ from (\ref{e:t-linear}).
\end{proof}

In the following Proposition we prove a more general statement than what is needed for the proof of Theorem \ref{t:weakSid}. Below we consider equations of arbitrary length with $\pm 1$ coefficients. For the theorem above, we just need to consider equations of length $t-1$ (which we recall may be assumed to be even): $\sum_{i=1}^{t-1}a_ix_i=0$, where $a_i \in \pm 1$. In this case, since $t-1$ is even, we have for parity reasons that $|\#\{a_i=1\} - \#\{a_i=-1\}|$  must be even, and so the condition that $t \ne (2m+1)|\#\{a_i=1\} - \#\{a_i=-1\}|$ from the statement of the proposition below is automatically satisfied. 

\begin{prop}\label{p:special}
Let $\Psi$ be a codimension $2$ linear system in $t$ variables which is linearly generic.  Let $\Phi$ be determined by the equation $\sum_{i=1}a_ix_i=0$, where each $a_i\in \{-1,1\}$ and $\#\{a_i=1\} \ne \#\{a_i=-1\}$. If $p$ is sufficiently large in terms of $t$ and $t\ne (2m+1)|\#\{a_i=1\} - \#\{a_i=-1\}|$ for any integer $m$, 
then there exists $f: \F_p \to [-1,1]$ with $\E_x f (x)=0$, with $T_{\Phi}(f) = 0$ and with $|T_\Psi(f)| \gg_{p,t} 1$.
\end{prop}
\begin{proof}
Let $V = \im M_\Psi^T \leq \F_p^t$. For $i=1,\ldots , t$, let $P_i$ be the $i$th coordinate hyperplane (i.e. $\{x \in \F_p^t : x_i = 0\}$).  Let $\Sym(t)$ the symmetric group on $t$ elements act on $\F_p^t$ by coordinate permutation, and let $(\Z/2\Z)^t$ act on $\F_p^t$ by reflection in each coordinate hyperplane (so for example $(1,0,1,0,\ldots)\in (\Z/2\Z)^t$ maps $(u_1,u_2,u_3,u_4,\ldots)$ to $(-u_1,u_2,-u_3,u_4,\ldots)$). Let $G$ be the group of permutations of $\F_p^t$ generated by $\Sym(t)$ and $(\Z/2\Z)^t$ and let $S = G - \{\id , (1,1,\ldots,1)\}$.  

We claim that $V - \left( \bigcup_{i=1}^t P_i \cup \bigcup_{\sigma \in S} \sigma V\right)$ is non-empty. Indeed,
\[\left| V - \left( \bigcup_{i=1}^t P_i \cup \bigcup_{\sigma \in S} \sigma V\right)\right| \geq |V| - \sum_{i=1}^t |V\cap P_i| - \sum_{\sigma\in S}|V\cap \sigma V|,\]
where each $V\cap P_i$ and $V\cap \sigma V$ is a subspace of $V$, so has size $1,p$ or $p^2$. By taking $p$ sufficiently large in terms of $t$, it suffices to show that $V \ne V\cap P_i$ and $V \ne \sigma V$ for each $i$ and $\sigma \in S$. That $V \ne V \cap P_i$ follows from the fact that $M$ cannot have a zero column, and that $V\ne \sigma V$ for every $\sigma \in S$ is an easy exercise in linear algebra, using the fact that each $2\times 2$ minor of $M$ has full rank and that every element of $G$ may be written in the form $\sigma \tau$ for some $\sigma \in \Sym(t)$ and $\tau \in (\Z/2\Z)^t$. 

Let $u :=(u_1,\ldots,u_t) \in V - \left( \bigcup_{i=1}^t P_i \cup \bigcup_{\sigma \in S} \sigma V\right)$. We may assume that $\#\{a_i = 1\} > \#\{a_i=-1\}$ so that by Fourier inverting
\begin{equation}\label{e:tsones-linear}
T_{\Phi} (f) = \sum_{h \in \F_p} |\hat f(h)|^{2k}\hat f(h)^l,
\end{equation}
where $k \geq 0$, $l>0$. We have by construction that each $\pm u_i$ is distinct and nonzero. Let $\frac{1}{2} \leq c_1,\ldots, c_t \leq 1$ be real numbers to be chosen later and define $\hat f(u_i):= c_i e^{\frac{\pi i}{2l}}/(2t)$ and $\hat f(-u_i) = \overline{\hat f(u_i)}$ for all $i$. Set $\hat f(h) = 0$ for all other $h$. Then $\E_xf(x) = \hat f(0) = 0$, $f$ takes values in $[-1,1]$ and $\Re (\hat f(\pm u_i)^l) = 0$ for all $i$.  Thus we have
\[ T_{\Phi} (f) = \Re\left( T_{\Phi} (f)\right) = \sum_{h \in \F_p^n} |\hat f(h)|^{2k}\Re(\hat f(h)^l)=0.\]
It remains to argue that $|T_\Psi(f)| \gg_{p,t} 1$. Fourier inverting we have 
\[ T_\Psi(f) = \sum_{y_1,y_2\in \F_p}\prod_{i=1}^t \hat f(b_{1i}y_1 + b_{2_i}y_2).\]
 Let $U=\{\pm u_1,\ldots, \pm u_t\}$ and let $\mathcal{I}$ be the collection of multisubsets of $[t]$ which have $t$ elements. For a multisubset $W$ of $U$ with $t$ elements, let $i(W)\in \mathcal{I}$ be the multisubset of indices in $W$ (for example if $t=4$ we may have $i([u_1,u_1,u_3,-u_1]) = [1,1,3,1]$), and let $\Sigma(W)$ be the sum of the signs in $W$ (so $\Sigma([u_1,u_1,u_3,-u_1]) = 1+1+1-1=2$). Then:
\begin{equation}\label{e:combin_tpsi}
T_\Psi(f) = \sum_{I\in \mathcal{I}} \left( \sum_{W:i(W)=I} \#\{y_1,y_2:M^T\binom{y_1}{y_2}=W\}\cdot \frac{1}{(2t)^t}e^{\frac{\Sigma(W)\pi i}{2l}}  \right) \prod_{i\in I}c_i,
\end{equation}
where we abuse notation by implicitly `unordering' the vector $M^T\binom{y_1}{y_2}$ and identifying it with the corresponding multiset. Viewing (\ref{e:combin_tpsi}) as a polynomial in $c_1,\ldots, c_t$, we may choose $c_1,\ldots, c_t$ in such a way to force $|T_\Psi(f)|\gg_{p,t} 1$ unless (\ref{e:combin_tpsi}) is the zero polynomial. Note that when $I=[t]$, we have by our construction of $u=(u_1,\ldots, u_t)$ that the only multisubsets $W$ for which $\#\{y_1,y_2:M^T\binom{y_1}{y_2}=W\}$ is nonzero are $W=[u_1,u_2,\ldots, u_t]$ and $W=[-u_1,-u_2,\ldots, -u_t]$, and in these cases $\#\{y_1,y_2:M^T\binom{y_1}{y_2}=W\}=1$. Furthermore, for these $W$ we have that $\Sigma(W) = t$ and $-t$ respectively. Therefore, the coefficient of the term $\prod_{i=1}^t c_i$ in (\ref{e:combin_tpsi}) is $\frac{2}{(2t)^t}\Re\left(e^{\frac{t\pi _i}{2l}}\right)$, which is nonzero since $t$ is not of the form $(2m+1)l$ for any integer $m$. Thus we may choose $c_1,\ldots, c_t$ so that $|T_\Psi(f)|\gg_{p,t} 1$, completing the proof.
\end{proof}

\begin{remark}\label{r:relax}
It is not difficult to adapt the construction from the previous Proposition to relax the assumption that $t \ne (2m+1)l$ to the assumption that $l\ne 1$ or $t$ is even; one proceeds by defining $\hat f(u_i):= c_i e^{\frac{(2n_i+1)\pi i}{2l}}$ for an appropriate choice of integers $\{n_i\}_{i=1}^t$. In fact, it seems likely that a modest perturbation of the above argument may allow us to remove any assumption of this kind completely. We don't pursue the matter in this document as we have no immediate need for the more general statement. We hope that the above example is illustrative.
\end{remark}

\subsection{Systems with additive tuples are not locally Sidorenko}

Recall that an additive $k$-tuple is a tuple $(x_1,\ldots, x_k)$ satisfying $x_1 - x_2 + \cdots + (-1)^{k+1}x_k = 0$. In this subsection we prove the following theorem. Together with Theorem \ref{t:weakSid}, this completes the proof of Theorem \ref{t:classification}.

\begin{thm}\label{t:localSid}
Let $\Psi$ be be a system of two linear equations which is linearly generic. Then $\Psi$ is not locally Sidorenko.
\end{thm}

The main observation underpinning the upcoming construction is that the true complexity of a system of linear forms is not maintained under Fourier inversion. We have not investigated the extent to which higher-order Fourier analysis may be applied in frequency space, but the upcoming construction (and natural generalisations thereof) at least demonstrates one application of the aforementioned observation.

In light of Theorem \ref{t:weakSid} we may assume that $t$ is odd. We note that if $\Psi(S)$ corresponds to an additive $(t-1)$-tuple, then by Fourier inversion, $T_{\Psi(S)}$ is the ($(t-1)$th power of the) $\ell^{t-1}$ norm in frequency space. In particular, if $T_{\Psi(S)}(f)=0$ then $f=0$, so the proof strategy from the previous subsection fails. We will have to settle for a function $f$ for which the ratio $T_{\Psi(S)}(f)/T_\Psi(f)$ is arbitrarily large. The existence of such a function is proven in Proposition \ref{p:quadratic}. First, we need a lemma.

\begin{lemma}\label{l:gauss-sums}
The following facts hold: 
\begin{enumerate}
\item Define the quadratic $q: \F_p^n \to \F_p$ by $q(x) = x^TMx + h^T x$, where $M$ is a matrix of rank $r$ and $h \in \F_p^n$. Then \[|\E_{x\in \F_p^n} e_p(q(x))|\leq p^{-r/2}.\] 
\item Let $\Psi$ be a system of linear forms, each mapping $(\F_p^n)^D$ to $\F_p^n$. Let $A \subset \F_p^n$ be the zero set of the quadratic defined by $x^Tx$. Then 
\[|T_\Psi(1_A) - p^{-k}| \leq p^{-n/2},\]
where $k$ is the dimension of the $\F_p$-span of the functions $\{\psi_i(x)^T\psi_i(x)\}_{i=1}^t$. 
\end{enumerate} 
\end{lemma}
\begin{proof}[Sketch proof.]
Part 1 is a very standard Gauss sum estimate; a proof may be found in \cite{Gre07}, for example. In \cite[Proof of Theorem 3.1]{GW10} it is shown (using part 1 as the main ingredient) that if a system of linear forms $\tilde \Psi:= (\tilde \psi_i)_{i=1}^l$ has $\{\tilde \psi_i(x)^T\tilde \psi_i(x)\}_{i\in [l]}$ linearly independent, then $|T_{\tilde \Psi}(1_A) - p^{-l}| \leq p^{-n/2}$. Let $S\subset [t]$ be a set such that $\{\psi_i(x)^T\psi_i(x)\}_{i\in S}$ is a basis for $\{\psi_i(x)^T\psi_i(x)\}_{i=1}^t$. Then $\prod_{i\in S}1_A(\psi_i(x)) = \prod_{i\in [t]}1_A(\psi_i(x))$ for all $x$, so $T_\Psi(1_A) = T_{\Psi(S)}(1_A)$ and our result then follows from the second sentence of this proof. 
\end{proof}

We now briefly introduce some notation. For a $2 \times k$ matrix $M$  with values in $\F_p$, let $M^{(2)}$ be the $3 \times k$ matrix defined by 
\[M^{(2)} := \begin{pmatrix}
\vect{r_1}^2 \\
\vect{r_2}^2 \\
\vect{r_1}\vect{r_2}
\end{pmatrix},\]
where $\vect r_1$ and $\vect r_2$ are the first and second rows of $M$ respectively, and where multiplication of vectors is conducted coordinatewise.

\begin{prop}\label{p:quadratic}
Let $t \geq 5$ be odd. Let $\Psi$ be a system of $t$ linear forms whose image has codimension $2$ and which is linearly generic. Let $\Phi$ be the system defined by an additive $(t-1)$-tuple.  Then there is a function $f:\F_p^n \to [-1,1]$ such that $\E_x f(x) = 0$, such that $|T_\Psi(f)| \gg_{p,t} p^{-n(t-4)/2}$ and such that $T_{\Phi}(f) < p^{-n(t-3)/2}$.  
\end{prop}
\begin{proof}
We construct $f$ by specifying its Fourier transform $\hat f: \F_p^n \to \C$. Let $A$ be the zero set of the quadratic form given by $x^Tx$.   
Define $\hat f(0) =0$ and $\hat f(h) = \frac{1}{p^{n/2}+1}(1_A(h)-p^{-1})$ otherwise. Firstly, $T_{\Phi}(f) = \sum_h |\hat f(h)|^{t-1} < p^n \cdot \left(p^{-n/2}\right)^{t-1} = p^{-n(t-3)/2}$.

It is clear that for all $h \in \F_p^n$ we have $\hat f(h) = \overline{\hat f(-h)}$, and so $f$ takes values in $\R$. Furthermore, for all $x\ne 0$:
\begin{align*}
|f(x)| = \left|\sum_h \hat f(h)e_p(x \cdot h)\right| &= \frac{1}{p^{n/2}+1}\left|\sum_h (1_A(h) - p^{-1})e_p(x \cdot h) - (1-p^{-1})\right| \\
&\leq  \frac{1}{p^{n/2}+1}\left(\left|p^n \E_h \left(\E_{a\in \F_p} e_p(ah^Th)\right) e_p(x^T h)\right| + (1-p^{-1})\right)\\
&\leq \frac{1}{p^{n/2}+1}\left(p^n\E_{a \in \F_p}\left| \E_h  e_p(ah^Th+x^T h)\right| + (1-p^{-1})\right)\\
&\leq \frac{1}{p^{n/2}+1}|p^{n/2} + 1-p^{-1}| \\
&< 1,
\end{align*}
by Fourier inverting in the first line, orthogonality of characters on $\F_p$ in the second line, and Lemma \ref{l:gauss-sums} part 1 in the fourth. For $x=0$ we have \[ f(0) = \frac{1}{p^{n/2}+1}\left( \sum_h (1_A(h) - p^{-1}) - 1+p^{-1}\right)  \leq 1,\]
by Lemma \ref{l:gauss-sums} part 2 with the trivial set of linear forms $x \mapsto (x)$. Thus $f$ takes values in $[-1,1]$.

By (\ref{e:T-inversion}) and the condition that the image of $\Psi$ has codimension 2 we have 
\[|T_\Psi(f)| = p^{2n}|T_{\Psi^\perp}(\hat f)| \gg_{p,t} p^{2n}\cdot p^{-tn/2}\left|T_{\Psi^\perp}(1_A-p^{-1}) + O_t(p^{-n/2})\right|,\]
where $\Psi^\perp$ denotes the dual set of linear forms which may be determined as in (\ref{e:T-inversion}).
It remains to show that $|T_{\Psi^\perp}(1_A-p^{-1})| \gg_{p,t} 1$. We have 
\begin{align*}
T_{\Psi^\perp}(1_A-p^{-1})  &= \sum_{S \subset [t]}(-p)^{-t+|S|}T_{\Psi^\perp(S)}(1_A)\\
&=\sum_{S\subset [t]}(-p)^{-t+|S|}\E_{x_1,x_2} \prod_{i\in S}1_0\left(\psi_i^\perp(x_1,x_2)^T\psi_i^\perp(x_1,x_2)\right).
\end{align*} 

Since the image of $\Psi$ has codimension 2, we have that each $\psi_i^\perp \in \Psi^\perp$ maps $(\F_p^n)^2 \to \F_p^n$. Thus each $\psi_i^\perp(x)^T\psi_i^\perp(x)$ maps $x=(x_1,x_2) \in (\F_p^n)^2$ to a linear combination of $\{x_1^Tx_1, x_1^Tx_2, x_2^Tx_2\}$, so the dimension of the span of the functions $\{\psi_i^\perp(x)^T\psi_i^\perp(x)\}_{i=1}^t$ is at most $3$. By Lemma \ref{l:gauss-sums}, 
\[ T_{\Psi^\perp}(1_A-p^{-1})  = \sum_{S \subset [t]}(-p)^{-t+|S|}p^{-\dim \spa \{\psi_i^\perp(x)^T\psi_i^\perp(x)\}_{i\in S}} + O_t(p^{-n/2}).\]
One observes that $\dim \spa \{\psi_i^\perp(x)^T\psi_i^\perp(x)\}_{i\in S} = \rank M^{(2)}_S$ where $M_S$ is the matrix obtained by restricting the columns of $M_\Psi$ to the coordinates in $S$. We show that $|T_\Psi(1_A-p^{-1})| \gg_{p,t} 1$ by showing that $\sum_{S \subset [t]}(-p)^{-t+|S|}p^{-\rank M^{(2)}_S}$ is nonzero. Using the fact that every $2\times 2$ minor of $M_{\Psi}$ has nonzero determinant, an easy exercise in linear algebra shows that for each $S \subset [t]$ with $|S| \geq 3$, $\rank M_S^{(2)} = 3$. Furthermore it is clear that if $|S|=1$, $2$ then $\rank M_S^{(2)} = 1$, $2$ respectively. Thus we have 
\begin{align*}
\sum_{S \subset [t]}(-p)^{-t+|S|}p^{-\rank M^{(2)}_S} &= p^{-3}\sum_{|S| \geq 3}(-p)^{-t+|S|} + (\binom{t}{2} - t + 1)(-p)^{-t} \\
&= \frac{(1-p)^{t} + p^3(\binom{t}{2} - t + 1) - (p^2\binom{t}{2} - pt + 1)}{(-p)^{-t}p^{3}}.
\end{align*}
We claim that the numerator is nonzero for pairs $(p,t)$ with $p$ prime and $t\geq 5$. This may be checked by direct computation for small values of $(p,t)$, and for larger values of $(p,t)$ it is clear that the term $(1-p)^t$ dominates. This completes the proof. 
\end{proof}

\begin{remark}
In the above proposition, the assumption that $\Psi$ is linearly generic may be removed quite easily, at least if one allows $p$ to be sufficiently large in terms of $t$. Indeed one writes $\sum_{S \subset [t]}(-p)^{-t+|S|}p^{-\rank M^{(2)}_S}$ as a $p$-adic series $\sum_{i=-t}^{-1} a_ip^i$. Each $|a_i|$ may be bounded in terms of $t$, so if $p$ is large enough in terms of $t$ then this series representation is unique and so the sum is zero if and only if each $a_i$ is zero. One may use some linear algebra to show that at least one of the $a_i$ must be nonzero. It is clear how the more general result may be used in the further classification of Sidorenko systems, but we have no need for such a statement in this document and so we do not write out the details.
\end{remark}

\begin{proof}[Proof of Theorem \ref{t:localSid}]
Recall that we may assume that $t$ is odd. As in the proof of Theorem \ref{t:weakSid}, using (\ref{e:lingen-expansion}), we have 
\[T_\Psi(\alpha + \eps f) = \alpha^t + \alpha \eps^{t-1}\sum_{|S|=t-1} T_{\Psi(S)}(f) + \eps^t T_\Psi(f).\]
As we have seen already, each $\Psi(S)$ has codimension one and so $T_{\Psi(S)}(f)$ is of the form $\sum_h \prod_{i=1}^t \hat f(a_i h)$. By H\"older's inequality then 
\[ |T_{\Psi(S)}(f)| \leq \sum_{h} |\hat f(h)|^{t-1} = T_\Phi(f),\] where we let $\Phi$ denote an additive $(t-1)$-tuple. Thus 
\[T_\Psi(\alpha + \eps f) \leq \alpha^t + (t-1)\alpha \eps^{t-1} T_{\Phi}(f) + \eps^t T_\Psi(f).\]
Let $f$ be the function from Proposition \ref{p:quadratic}, and if $T_\Psi(f) > 0$ then replace $f$ with $-f$, so 
\[T_\Psi(\alpha + \eps f) \leq \alpha^t + (t-1)\alpha \eps^{t-1} p^{-n(t-3)/2} - \eps^t c_{p,t} p^{-n(t-4)/2}.\] For $\eps, \alpha$ fixed, we may choose $n$ large enough so that $T_\Psi(\alpha + \eps f) < \alpha^t$, completing the proof.
\end{proof} 

This also completes the proof of Theorem \ref{t:classification}, which is a union of Theorem \ref{t:weakSid} and Theorem \ref{t:localSid}.

\begin{maincor-repeat}
Conjecture \ref{c:KLM5sid} is true and the answer to Question \ref{q:KLMkSid} is no. 
\end{maincor-repeat}
\begin{proof}
Recall that $s(\Psi)$ is the size of the smallest subsystem of $\Psi$ whose image does not contain a coordinate hyperplane. An easy exercise in linear algebra shows that $s(\Psi) = t-1$ implies that $\Psi$ is linearly generic. 
\end{proof}

\section{Examples}\label{s:examples}

\subsection{Locally Sidorenko, but not Sidorenko}\label{ss:local-not-sid}

Recall that a system $\Psi$ comprising $t$ linear forms is locally Sidorenko if for all $\alpha \in [0,1]$ there exists $\eps >0$ such that for all $n\geq 1$ and all $f : \F_p^n \to [-\alpha, 1-\alpha]$ with $\E f = 0$ we have $T_\Psi(\alpha + \eps f) \geq \alpha^t$. We begin by providing the following example of a system which is locally Sidorenko but not Sidorenko.  This proves Proposition \ref{p:local-not-sid}.

\begin{example}\label{ex:sid-not-local}
Let $\Phi$ denote the following system of two linear equations in nine variables: 
\begin{align*}
x_1 - x_2 + x_3 -x_4 &= 0\\
x_5 -x_6 + x_7 - x_8 + x_9 &=0.
\end{align*}
We firstly observe that $\Phi$ is not Sidorenko. Indeed, let $A = \{ 1\}\subset \F_p$. Then $\E_{x \in \F_p}1_A(x) = 1/p$, but $A$ contains no solutions to $\Phi$, so $T_\Phi(1_A)=0$.

We claim that $\Phi$ is locally Sidorenko. Let $\Af$ be shorthand for the first equation in $\Phi$ and $\Ai$ for the second. We observe that  for a  function $f$ on $\F_p^n$ we have $T_\Phi(f) = T_{\Af}(f)T_{\Ai}(f)$. Also, since $\Af$ has codimension one, we get $T_{\Af}(\alpha + f)=\alpha^4 + T_{\Af}(f)$ for functions $f$ with $\E f=0$, and similarly for $\Ai$. Thus, for $f:\F_p^n \to [-\alpha, 1-\alpha]$ with $\E f=0$,
\begin{align*}
T_\Phi(\alpha + \eps f) &= (\alpha^4+T_{\Af}(\eps f))(\alpha^5+T_{\Ai}(\eps f)) \\
&= \alpha^9 + \eps^4\left(\alpha^5 T_{\Af}(f) + \alpha^4\eps T_{\Ai}(f) + \eps^5T_{\Af}(f)T_{\Ai}(f)\right).
\end{align*}
Fourier-inverting, we have that $T_{\Af}(f) = \sum_h |\hat f(h)|^4$ and $T_{\Ai}(f)=\sum_{h}|\hat f(h)|^4 \hat f(h)$. In particular we have that $|T_{\Ai}(f)| \leq T_{\Af}(f) \leq 1$. Thus,
\[ T_\Phi(\alpha + \eps f) \geq \alpha^9 + \eps^4T_{\Af}(f) \left( \alpha^5 - \alpha^4 \eps - \eps^5\right),\]
and so setting $\eps \leq \alpha/2$ we have that $T_\Phi(\alpha + \eps f)\geq \alpha^9$, so $\Phi$ is indeed locally Sidorenko. 
\end{example}

\subsection{Weakly locally Sidorenko, but not locally Sidorenko}\label{ss:weak-not-local}
Recall that a system $\Psi$ comprising $t$ linear forms is weakly locally Sidorenko if for all $n\geq 1$, for all $\alpha \in [0,1]$ and all $f:\F_p^n \to [-\alpha,1-\alpha]$ with $\E f =0$, there exists $\eps_f >0$ such that for all $\eps < \eps_f$ we have $T_\Psi(\alpha + \eps f) \geq \alpha^t$. In this subsection we demonstrate a linearly generic system which is weakly locally Sidorenko, but not locally Sidorenko, proving Corollary \ref{cor:weak-not-local}.  The example is taken from \cite[Example 4.6]{KLMSid}, as is most of the analysis towards showing that it is weakly locally Sidorenko. That it is not locally Sidorenko follows from Theorem \ref{t:classification}. In the following example we assume that $p \not \in \{2,3\}$.

\begin{example}\label{ex:weak-not-local}
Let $\Phi$ be the following system of two equations in five variables: 
\begin{align*}
x_1 - x_2 + x_3 - x_4 &=0\\
x_1 + 2x_2 -x_3 - 2x_5 &=0.
\end{align*}
One sees that $\Phi$ is linearly generic. By Theorem \ref{t:classification}, $\Phi$ is not locally Sidorenko. We claim that $\Phi$ is weakly locally Sidorenko. For $f$ with $\E f=0$ we have
\[ T_\Phi(\alpha + \eps f) = \alpha^5 + \alpha \eps^4 \sum_{|S|=4}T_{\Phi(S)}(f) + \eps^5T_\Phi(f),\]
so it suffices to show that $\sum_{|S|=4}T_{\Phi(S)}(f)>0$ whenever $f \ne 0$. One computes using (\ref{e:T-inversion}) and recalling that $T_{\Phi(S)}(f)$ is real-valued that
\begin{align*}
\sum_{|S|=4}T_{\Phi(S)}(f) &= \sum_{h} |\hat f(h)|^4 + 2|\hat f(h)|^2|\hat f(2h)|^2 +  2\hat f(-h)\hat f(2h)^2 \hat f(-3h) \\
&= \sum_{h} \frac{1}{2}|\hat f(2h)|^4+ \frac{1}{2}|\hat f(3h)|^4+ 2|\hat f(h)|^2|\hat f(2h)|^2 +  2\Re \left(\hat f(-h)\hat f(2h)^2 \hat f(-3h)\right)\\
&\geq \sum_h |\hat f(-2h)\hat f(3 h)|^2 + 2|\hat f(h)|^2|\hat f(2h)|^2 +  2\Re \left(\hat f(-h)\hat f(2h)^2 \hat f(-3h)\right)\\
&= \sum_h \left|\hat f(h)\hat f(-2h) + \overline{\hat f(-2h)\hat f(3h)}\right|^2 +  |\hat f(h)|^2|\hat f(2h)|^2\\
&\geq  \sum_h |\hat f(h)|^2|\hat f(2h)|^2.
\end{align*}
This is positive unless $|\hat f(h)|^2|\hat f(2h)|^2=0$ for all $h$. But in this case  $\sum_{|S|=4}T_{\Phi(S)}(f) = \sum_h |\hat f(h)|^4$, which is positive unless $f$ is zero. This completes the proof of the claim that $f$ is weakly locally Sidorenko.
\end{example}

\subsection{Locally Sidorenko, but no additive tuple}

Next, we demonstrate that the conclusion of Theorem \ref{t:classification} cannot hold if one removes the condition that the equations are linearly generic. In particular, there exist systems which are not linearly generic, which do not contain an additive tuple and which are locally Sidorenko.

 Recall that if $\Phi$ comprises only an additive $(2k)$-tuple, the corresponding functional $T_{\Phi}$ is a norm on the space of functions and in particular is positive definite: $T_{\Phi}(f) = 0 \implies f =0 \implies T_{\Psi}(f)=0$, for any $\Psi$. This is the main obstruction to the proof of Theorem \ref{t:weakSid} going through for systems with an additive $(2k)$-tuple. It transpires that there exists a non-linearly generic $\Psi$ with a subsystem $\Phi$ whose corresponding functional is not positive definite but with the property that $T_{\Phi}(f) = 0 \implies T_{\Psi(S)}(f)=0$ for all subsystems $\Psi(S) \subseteq \Psi$. This motivates the upcoming example.

\begin{example}\label{ex:no-add-tup}
Let $\Phi$ be the following system of two equations in eight variables: 
\begin{align*}
x_1 - x_2 + 2x_3 - 2x_4 + x_5 + x_6 + x_7 + x_8 &=0\\
x_5-x_6+2x_7-2x_8&=0.
\end{align*}
We claim that $\Phi$ is  locally Sidorenko (and it clearly does not contain any additive $(2k)$-tuples). 
Fix $\alpha \in (0,1)$ and let $f: \F_p^n \to [-\alpha, 1-\alpha]$.  Invoking  (\ref{e:T-inversion}) and with some computation one finds that
\begin{align*}
T_\Phi(\alpha + \eps f) = \alpha^8 &+ \alpha^4\eps^4 \sum_h |\hat f(h)|^2|\hat f(2h)|^2 \\ 
&+ \alpha\eps^7 \sum_h |\hat f(h)|^2|\hat f(2h)|^2 \hat f(-2h)\hat f(h)\hat f(-3h)\\
&+ \alpha \eps^7 \sum_h |\hat f(h)|^2|\hat f(2h)|^2 \hat f(2h)\hat f(3h)\hat f(-h)\\
&+ \alpha \eps^7 \sum_h |\hat f(h)|^2|\hat f(2h)|^2 \hat f(-h)\hat f(-3h)\hat f(-4h)\\
&+ \alpha \eps^7 \sum_h |\hat f(h)|^2|\hat f(2h)|^2 \hat f(3h)\hat f(h)\hat f(4h)\\
&+ \eps^8 \sum_{h_1,h_2} |\hat f(h_1)|^2|\hat f(2h_1)|^2 \hat f(h_1+h_2)\hat f(h_1-h_2)\hat f(h_1+2h_2)\hat f(h_1-2h_2).
\end{align*}
The absolute value of the final sum is at most 
\[ \sum_{h_1} |\hat f(h_1)|^2|\hat f(2h_1)|^2 \sum_{h_2}\left|\hat f(h_1+h_2)\hat f(h_1-h_2)\hat f(h_1+2h_2)\hat f(h_1-2h_2)\right|\] \[\qquad \qquad \qquad \qquad \qquad \qquad \qquad \qquad \leq \sum_{h_1} |\hat f(h_1)|^2|\hat f(2h_1)|^2 \sum_{h_2}\left|\hat f(h_1+h_2)\hat f(h_1-h_2)\right|,\]
by the bound $||\hat f||_\infty \leq 1$. Applying Cauchy--Schwarz to the sum over $h_2$ and subsequently Parseval's identity yields an upper bound of $\sum_{h}|f(h)|^2|f(2h)|^2$. Thus
\[ T_\Phi(\alpha + \eps f) \geq \alpha^8 + \eps^4 \left(\sum_h |\hat f(h)|^2 |\hat f(2h)|^2\right)\left(\alpha^4 -  4\alpha \eps^3 - \eps^4\right).\]
Thus setting $\eps \leq \alpha/2$, we have that $T_\Phi(\alpha + \eps f)\geq \alpha^8$, and so $\Phi$ is locally Sidorenko.
\end{example}

We note that the previous example is not Sidorenko, so gives another, albeit less simple, example verifying Proposition \ref{p:local-not-sid}.

\section{Localising to subsystems with fewest variables}\label{s:local-4ap}

In this section we observe how to `localise' to subsystems with the fewest variables. As an application we outline a relatively short proof of the recent result of Kam\v cev--Liebenau--Morrison \cite{KLMCom} and independently Versteegen \cite{V4ap} that any system containing a four term arithmetic progression (or indeed any system of two independent equations in four variables) is uncommon. Our proof does not use any quadratic methods, which is also new here. 

Previously we have localised by considering small perturbations of a constant function $\alpha + \eps f$, whereupon, defining $s$ to be the  size of smallest subsystem of $\Psi$ for which some $T_{\Psi(S)}(f)$ is nonzero, we have 
\[T_\Psi(\alpha + \eps f) \approx \alpha^t + \eps^{s} \sum_{|S| = s}T_{\Psi(S)}(f).\]

For fixed positive integers $n$, $k$ and $f:\F_p^n \to [-1,1]$, define the function $f^{(k)}:\F_p^{n+k}\to [-1,1]$ by $f^{(k)} := f\otimes 1_0$ where $1_0$ is the characteristic function for $0\in \F_p^k$ (that is, $f^{(k)}=(f\circ \pi_1)(1_0\circ \pi_2)$, where $\pi_1$ projects onto the first $n$ coordinates, $\pi_2$ projects onto the final $k$ coordinates). Then for a system $\Psi=(\psi_1,\ldots,\psi_t)$ in $D$ free variables, we have recalling (\ref{e:tensor})  that
\[ T_\Psi(f^{(k)}) = T_\Psi(f)T_\Psi(1_0) = p^{-kD} T_\Psi(f).\]
Thus letting $k$ be large we have that 
\[ T_\Psi(\alpha + f\otimes 1_0) \approx \alpha^t + p^{-kd}\sum_{D(\Psi(S)) = d}T_{\Psi(S)}(f),\]
where $D(\Psi(S))$ denotes the number of variables in $\Psi(S)$ and $d$ is the minimal number of variables in a system for which $T_{\Psi(S)}(f)$ is nonzero. 

Now we outline an alternative proof of the fact that any system containing a $4$AP is uncommon. In fact we will prove the stronger result (as \cite{KLMCom} did) that any system containing a rank two system  with four linear forms is uncommon. We note that some parts of the argument remain the same. In particular we need the result that rank 2 systems in 4 variables can have $T_\Psi(f)< 0$. Rather than reproving it we just quote the following weak version of \cite[Lemma 4.1]{KLMCom} whose proof we note does not use higher-order methods. 

\begin{lemma}\cite[Lemma 4.1]{KLMCom}\label{l:klm41}
Let $(\Psi_i)_{i=1}^k$ be a collection of linear systems each in four variables with codimension two. Suppose that $s(\Psi_i)=3$ for all $i$.\footnote{Recall that this means that every subsystem of $\Psi_i$ containing two linear forms has image containing a coordinate hyperplane.} Then there exists $f: \F_p^n \to [-\frac{1}{2}, \frac{1}{2}]$ with $\E f =0$ such that $T_{\Psi_i}(f) < 0$ for $i=1,\ldots, k$. 
\end{lemma}

\begin{thm}[{\cite[Theorem 1.1]{KLMCom}, generalisation of \cite[Theorem 1.4]{V4ap}}]\label{t:4ap}
Any linear system of distinct, nonzero forms which contains a dimension 4 subsystem of rank 2 is uncommon.
\end{thm}
\begin{proof}
Let $t$ be the number linear forms in $\Psi$ and let $f:\F_p^n \to [-\frac{1}{2},\frac{1}{2}]$ have $\E f=0$ and be otherwise unspecified for the meantime. Let $k\geq 0, \eps = p^{-k}$ and $1_0:\F_p^k \to \{0,1\}$ be the characteristic function of zero. Then 
\[ T_\Psi\left(\frac{1}{2} + (\eps f)\otimes 1_0\right) + T_\Psi\left(\frac{1}{2}-(\eps f)\otimes 1_0\right) = \frac{1}{2^{t-1}} + \sum_{\substack{\varnothing \ne S \subset [t] \\ |S| \text{ even}}}\frac{1}{2^{t-|S|-1}}p^{-k(|S|+D(\Psi(S)))}T_{\Psi(S)}(f).\]
Letting $k$ be arbitrarily large, we see it suffices to show that there is a choice of $f$ such that $\sum_{S:D(S)+|S| \text{ minimal}} T_{\Psi(S)}(f) < 0$. Note that this minimum is at most $6$ since $\Psi$ contains a system of four linear forms with codimension two. If $D(\Psi(S)) = |S|$ then we average over a coordinate hyperplane and get $T_{\Psi(S)}(f) = 0$ since $\E f = 0$. Thus we may assume that $D(\Psi(S))<|S|$. 

If the minimum is equal to $6$, then the minimal systems $\Psi(S)$ must have $|S|=4, D(\Psi(S))=2$. If this is the case then we must have that $s(\Psi(S))=3$ for all of these minimal systems, or else they would give rise to a further subsystem $S'$ with $|S'| + D(\Psi(S'))$ smaller. Thus we may conclude by invoking Lemma \ref{l:klm41} in this case. 

Otherwise any $\Psi(S)$ contributing to the minimum must have $D(\Psi(S))=1$, and indeed cannot have $|S|=4$ since any system $\Psi(S)$ with $|S|=4, D=1$ gives rise to $\binom{4}{2}$ subsystems $\Psi(S_i)$ with $D(\Psi(S_i))=1, |S_i|=2$. In the remaining case that the minimal systems have $D(\Psi(S))=1, |S|=2$ we note that these systems are described by a single equation which cannot have its two coefficients summing to zero since this would mean that $\Psi$ has a pair of repeated forms. Thus we may conclude by the Fox--Pham--Zhao random Fourier sampling argument to find $f$ such that $\sum_{S:D(S)+|S| \text{ minimal}} T_{\Psi(S)}(f) < 0$.\footnote{In fact this is a very easy case of the Fox--Pham--Zhao argument since the equations comprise only two variables. It is likely an even simpler argument exists, but this is a known (and nice) argument, so we quote it for brevity.} This completes the proof.
\end{proof}

\begin{remark}
Kam\v cev--Liebenau--Morrison conjecture \cite[Conjecture 6.1]{KLMCom} that an analogue of Lemma \ref{l:klm41} holds for systems comprising an even number of linear forms whose image has codimension two. They note that their proof of \cite[Theorem 1.1]{KLMCom} however does not generalise to the case of higher number of variables, even if \cite[Conjecture 6.1]{KLMCom} is confirmed. We note that our proof above would generalise should \cite[Conjecture 6.1]{KLMCom} be confirmed. This would go some way towards the classification of weakly locally common systems.
\end{remark}

\bibliographystyle{alpha}
\bibliography{sid2}

\end{document}